\documentclass[leqno,12pt]{amsart}

\usepackage{palatino}
\usepackage{ytableau}
\usepackage[mathcal]{euler}

\usepackage{amssymb,amsmath,amsthm,enumerate,slashed}
\usepackage[latin1]{inputenc}
\usepackage[T1]{fontenc}
\usepackage{xcolor}
\usepackage{mathdots}
\usepackage{soul} 
\usepackage{colonequals}

\usepackage{tikz-cd}
\usepackage[left=1in, bottom=1in, right=1in, top=1in, paperwidth=8.5in, paperheight=11in]{geometry}

\usepackage[colorlinks=true, linkcolor=cyan, urlcolor= cyan, citecolor=blue]{hyperref}

\usepackage{graphicx}
\usepackage{epstopdf,epsfig}

\usepackage{xy}	
\xyoption{all}
\usepackage{marginnote}
\usepackage[ddmmyyyy]{datetime}

\swapnumbers
\numberwithin{equation}{section}

\theoremstyle{plain}
\newtheorem*{theorem*}{Theorem}
\newtheorem*{lemma*}{Lemma}

\newtheorem{theorem}[equation]{Theorem}

\newtheorem{proposition}[equation]{Proposition}
\newtheorem{corollary}[equation]{Corollary}

\theoremstyle{definition}
\newtheorem{definition}[equation]{Definition}
\newtheorem*{definition*}{Definition}

\newtheorem{question}[equation]{Question}
\newtheorem{notation}[equation]{Notation}
\newtheorem*{notation*}{Notation}

\newtheorem*{remark*}{Remark}

\title{Chromatic quasisymmetric functions of the path graph}
\author[Aliniaeifard, Asgarli, Esipova, Shelburne, van Willigenburg, Whitehead McGinley]{Farid Aliniaeifard, Shamil Asgarli, Maria Esipova, \\  Ethan Shelburne, Stephanie van Willigenburg, Tamsen Whitehead McGinley}
\date{\today}

\address{Farid Aliniaeifard: University of British Columbia, Vancouver, BC, Canada}
\email{farid@math.ubc.ca}
\address{Shamil Asgarli: Santa Clara University, Santa Clara, CA, USA}
\email{sasgarli@scu.edu}
\address{Maria Esipova: University of British Columbia, Vancouver, BC, Canada}
\email{mesipova@math.ubc.ca}
\address{Ethan Shelburne: University of British Columbia, Vancouver, BC, Canada}
\email{emshelburne@math.ubc.ca}
\address{Stephanie van Willigenburg: University of British Columbia, Vancouver, BC, Canada}
\email{steph@math.ubc.ca}
\address{Tamsen Whitehead McGinley: Santa Clara University, Santa Clara, CA, USA}
\email{tmcginley@scu.edu}

\thanks{The first, third, fourth, and fifth authors were supported in part by the Natural Sciences and Engineering Research Council of Canada}

\begin{document}

\begin{abstract} We show that the chromatic quasisymmetric function (CQF) of a labeled path graph on $n$ vertices is \emph{not} symmetric unless the labeling is the natural labeling $1, 2, ..., n$ or its reverse $n, ..., 2, 1$. We also show that the star graph $K_{1, n-1}$ with $n\geq 3$ has a nonsymmetric CQF for all labelings. 
\end{abstract}

\maketitle

\section{Introduction}

R. Stanley \cite{Stan95} defined the chromatic symmetric function of a graph as an infinite power series generalization of the chromatic polynomial: if a graph $G$ has vertex set $V=\{v_1, \ldots, v_n\}$, with $\mathcal{PC}(G)$ denoting its proper colorings using colors from $\mathbb P=\{1,2, \dots\}$, then its \emph{chromatic symmetric function} is 
$$
X(G; \mathbf{x}) = \sum_{c\in\mathcal{PC}(G)} \mathbf{x}^{c},
$$
where $\mathbf{x}^c = x_{c(v_1)} x_{c(v_2)}\cdots x_{c(v_n)}$ with $c(v_i)$ the color of vertex $v_i$. 

The chromatic quasisymmetric function of a graph is a generalization of the chromatic symmetric function. It was introduced by Shareshian and Wachs \cite{SW16} as a tool to solve the Stanley-Stembridge Conjecture \cite{StanStem} given in terms of the chromatic symmetric function in \cite{Stan95}, which is a central area of research in algebraic combinatorics. Special cases of this conjecture have been solved, for example \cite{ChoHuh, GebSag, MM, HuhNamYoo}. More recently, Hikita has a preprint \cite{H24} proving the Stanley-Stembridge conjecture in full generality. For the chromatic quasisymmetric function, we have a graph $G=(V,E)$, with $|V|=n$, whose vertices are labeled with $[n] = \{1, 2, ..., n\}$. More precisely, we have a bijective map ${L}\colon V\to [n]$ that assigns a \emph {label} to each $v\in V$. The data $(G, {L})$ constitute a \emph{labeled graph.} Instead of differentiating between the vertices and the labels, we will often identify the vertex set $V=\{v_1, v_2,. \ldots, v_n\}$ with $[n]=\{1, 2, \ldots ,n$\}.

Given a coloring $c\in \mathcal{PC}(G)$, the \emph{ascent set} of $c$ is:
$$
\operatorname{Asc}_G(c) = \{ij \in E \ | \ i<j \text{ and } c(i) < c(j) \}. 
$$
The \emph{ascent number} of $c$ is $\operatorname{asc}_G(c) = \# \operatorname{Asc}_G(c)$. Similarly, the \emph{descent set} of $c$ is $$\operatorname{Des}_G(c)=\{ij \in E \ | \ i<j \text{ and } c(i) > c(j) \}.$$ The \emph{descent number} of $c$ is $\operatorname{des}_G(c) = \# \operatorname{Des}_G(c)$. If the graph $G$ is clear from the context, we omit the subscript $G$. As above, we denote by $\mathbf{x}^{c}$ the monomial $x_{c(1)} x_{c(2)}\cdots x_{c(n)}$ where $c(i)$ is the color of the vertex labeled $i\in [n]$.

The \emph{chromatic quasisymmetric function (CQF)} of $G$ is 
$$
X(G; \mathbf{x}, q) = \sum_{c\in\mathcal{PC}(G)} \mathbf{x}^{c} q^{\operatorname{asc}(c)}.
$$
We can express $X(G; \mathbf{x}, q)$ as 
$$
X(G; \mathbf{x}, q)=\sum_{k\geq 0} X_k(G; \mathbf{x}) q^k,
$$ 
where each $X_k(G; \mathbf{x})$ is a quasisymmetric function of degree $n$. Moreover, $X(G; \mathbf{x}, q)$ is a polynomial in $q$ with degree equal to the number of edges of $G$. Finally, the chromatic quasisymmetric function refines the chromatic symmetric function as $X(G; \mathbf{x}, 1)=X(G; \mathbf{x})$. See \cite[Page 502]{SW16}.

In general, there is no known classification of labeled graphs whose chromatic quasisymmetric functions have coefficients that are symmetric. In this paper, we completely resolve this question for the path graph $P_n$. We also discuss the star graph $K_{n-1, 1}$ and give a sufficient (but not necessary) condition for an arbitrary tree to have a nonsymmetric CQF (see Corollary~\ref{cor:nonsymmetric-tree}).

The \emph{natural labeling} of $P_n$ is $1$, $2$, ..., $n$ or its reverse $n, ..., 2, 1$, since reversing labels on the path graph leads to an isomorphic labeled graph. More precisely, if $(P_n, L)$ is a labeled path, then it has a natural labeling if $L(v_i)=i$ for all $i$ or $L(v_i)=n+1-i$ for all $i$. For example, the two isomorphic natural labelings for $P_5$ are:
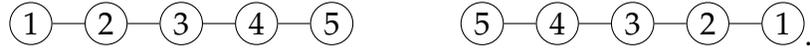
\begin{center}
\begin{tikzpicture}
    \foreach \i/\label in {1/1, 2/2, 3/3, 4/4, 5/5} {
        \node[draw, circle, inner sep=2pt] (v1\i) at (\i,0) {\label};
    }
    
    \foreach \i in {1,...,4} {
        \draw (v1\i) -- (v1\the\numexpr\i+1\relax);
    } 

    \foreach \i/\label in {1/5, 2/4, 3/3, 4/2, 5/1} {
        \node[draw, circle, inner sep=2pt] (v2\i) at (\i+6,0) {\label}; 
    } 
    
    \foreach \i in {1,...,4} {
        \draw (v2\i) -- (v2\the\numexpr\i+1\relax);
    }
    
\end{tikzpicture}.
\end{center}
The path graph below has a ``non-natural'' labeling:
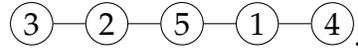
\begin{center}
\begin{tikzpicture}

    \foreach \i/\label in {1/3, 2/2, 3/5, 4/1, 5/4} {
        \node[draw, circle, inner sep=2pt] (v1\i) at (\i,0) {\label};
    }
    
    \foreach \i in {1,...,4} {
        \draw (v1\i) -- (v1\the\numexpr\i+1\relax);
    }
    \end{tikzpicture}.
\end{center}

Shareshian and Wachs \cite{SW16} proved that $X(P_n; \mathbf{x}, q)$ is symmetric if $P_n$ is labeled with the natural labeling. In this paper, we prove the converse statement, which leads to the following result:

\begin{theorem}\label{thm:main} The chromatic quasisymmetric function of $P_n$ is symmetric if and only if $P_n$ has the natural labeling.
\end{theorem}

The principal strategy is to associate to each labeled path graph $P_n$ a \emph{ribbon diagram} with $n$ boxes, which encodes the \emph {ascent-descent pattern} linked with the labeling of the path. Section~\ref{sec:ribbon} provides details about this approach.

\textbf{Outline.} In Section~\ref{sec:prelim}, we give  essential definitions and provide background. In Section~\ref{sec:ribbon}, we associate a ribbon tableau $RD(P_n)$ to a given colored labeled path. We prove our main Theorem~\ref{thm:main} in Section~\ref{sec:path}. Finally, in Section~\ref{sec:other-trees}, we treat the CQF of the star graph $K_{1, n-1}$ and present a general result for a class of trees. 

\medskip 

\textbf{Acknowledgments.} We are grateful to the referees for their detailed comments, subtle corrections, and pertinent suggestions, and to Maria Gillespie, Joseph Pappe, and Kyle Salois for helpful discussions.

\section{Preliminaries}\label{sec:prelim}

This section reviews the necessary definitions and background material for chromatic quasisymmetric functions (CQFs). We discuss both symmetric and palindromic CQFs. As a sufficient test for a CQF to be palindromic, we study the properties of the \emph{flip map} on the vertex set, given by $i\mapsto n+1-i$.

Following the development and notation in \cite{Sag20}, let $\mathbf{x}=\{x_1, x_2, \dots\}$ be a countably infinite set of commuting variables, and let $\mathbb C[[\mathbf{x}]]$ be the algebra of formal power series in $\mathbf{x}$.  We say $f(x) \in \mathbb C[[\mathbf{x}]]$ is \emph{homogeneous of degree n} if the sum of the exponents of each monomial in $f(x)$ equals $n$, and $f(x)$ is \emph{symmetric} if the coefficient of the monomial  $x_{i_1}^{m_1}x_{i_2}^{m_2}{\cdots} x_{i_\ell}^{m_\ell}$ (for $i_1, i_2,\ldots {,i_\ell}$ distinct subscripts) is the same as the coefficient of $x_{1}^{m_1}x_{2}^{m_2}{\cdots} x_{\ell}^{m_\ell}$ in $f(x)$.

Let $$\operatorname{Sym_n}= \{ {f(x)} \in \mathbb C[[\mathbf{x}]]  \ | \ {f(x)} \text{ is symmetric and homogeneous of degree }n\}.$$
Then \emph{the algebra of symmetric functions} is $$\operatorname{Sym}=\bigoplus_{n\ge 0}\operatorname{Sym}_n.$$
$\operatorname{Sym}$ is well-studied, and there are several interesting bases indexed by integer partitions.  The reader is referred to \cite{Sag20}, \cite{Stan23}, and \cite{Mac15} for a wealth of information on symmetric functions.

We say that $f(x)\in\mathbb C[[\mathbf{x}]]$ is \emph{quasisymmetric} if the coefficient of $x_{i_1}^{m_1}x_{i_2}^{m_2}{\cdots} x_{i_\ell}^{m_\ell}$ for $i_1< i_2<\dots <i_\ell$ is the same as the coefficient of $x_{1}^{m_1}x_{2}^{m_2}{\cdots} x_{\ell}^{m_\ell}$ in $f(x)$.
Let $$\operatorname{QSym_n}= \{ {f(x)}\in \mathbb C[[\mathbf{x}]]  \ | \ {f(x)} \text{ is quasisymmetric and homogeneous of degree }n\}.$$
Then the \emph{algebra of quasisymmetric functions} is $$\operatorname{QSym}=\bigoplus_{n\ge 0}\operatorname{QSym}_n.$$
Note that every symmetric function is also quasisymmetric, but the converse is false.  Bases of $\operatorname{QSym}$ are indexed by integer compositions.  

Given a composition $\alpha=(\alpha_1, \alpha_2, {\ldots}, \alpha_{\ell})\vDash n$, the \emph{monomial quasisymmetric function} $M_{\alpha}\in \operatorname{QSym}_n$ is defined by:
$$
M_{\alpha} = \sum_{i_1<i_2<\cdots < i_{\ell}} x_{i_1}^{\alpha_1} {x_{i_2}^{\alpha_2}}\cdots x_{i_{\ell}}^{\alpha_{\ell}}.
$$
Then, $\operatorname{QSym}_{n}$ is a vector space of dimension $2^{n-1}$, with basis $\{M_{\alpha}\ | \ \alpha \vDash n\}$ for $n\geq 1$ and $\operatorname{QSym}_{0}$ is spanned by $\{ M_{0} =1 \}$, where $0$ stands for the empty composition.

Given $\alpha = (\alpha_1, \alpha_2,\dots, \alpha_\ell) \vDash n$, the \emph{reverse} of $\alpha$ is the composition $\alpha^r = (\alpha_\ell, \alpha_{\ell-1},\dots, \alpha_1)$. Define an involution $\rho$ on $\operatorname{QSym}_{n}$ by assigning
$$\rho(M_\alpha) = M_{\alpha^r},$$
on the basis elements and extending linearly. The map $\rho$ further extends to an involution of $\operatorname{QSym}_{n}[q]$.

In this paper, we focus on two special classes of chromatic quasisymmetric functions.
\begin{itemize}
\item We say that a CQF $X(G;\mathbf{x}, q)$ is \emph{symmetric} if $X(G;\mathbf{x}, q) \in \operatorname{Sym}[q]$, meaning that the coefficients are symmetric functions (not just quasisymmetric functions). 
\item We say that a CQF $X(G;\mathbf{x}, q)$ is \emph{palindromic} if $X(G;\mathbf{x}, q)$ is a palindromic polynomial in $q$, meaning that the coefficient (living in $\operatorname{QSym}$) of $q^i$ equals the coefficient of $q^{m-i}$ where $m$ is the number of edges in $G$. 
\end{itemize}

An equivalent definition for symmetry of $X(G;\mathbf{x}, q)$ is that the coefficient of $M_\alpha$ in $X(G;\mathbf{x}, q)$ is equal to the coefficient of $M_\beta$ (as polynomials in $q$) where $\beta$ is any rearrangement of $\alpha$.

Shareshian and Wachs \cite[Proposition 2.6]{SW16} show that $\rho$ reverses the coefficients of a CQF as a polynomial in $q$, which we will use in Proposition \ref{prop:sum_Q_Q'}. 
Hence, a CQF $Q$ is palindromic if and only if $\rho(Q)=Q$. It is straightforward to check that a necessary condition for $Q$ to be palindromic is that the coefficient of $M_\alpha$ in $Q$ is equal to the coefficient of $M_{\alpha^r}$ for each composition $\alpha$. 

As an illustration, consider the labeled path graph $P_4$:
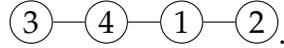
\begin{center}
\begin{tikzpicture}

    \foreach \i/\label in {1/3, 2/4, 3/1, 4/2} {
        \node[draw, circle, inner sep=2pt] (v1\i) at (\i,0) {\label};
    }
    
    \foreach \i in {1,...,3} {
        \draw (v1\i) -- (v1\the\numexpr\i+1\relax);
    }
    \end{tikzpicture}.
\end{center}
The associated CQF is
\begin{align*}
&(5q^3+7q^2+7q+5)M_{(1, 1, 1, 1)} + (2q^3+q^2+q+2)M_{(1, 1, 2)}+ \\&(q^3+2q^2+2q+1)M_{(1, 2, 1)} + (2q^3+q^2+q+2)M_{(2, 1, 1)} + (q^3+1)M_{(2, 2)},
\end{align*}
which is palindromic, but not symmetric since the coefficients of $M_{(1,1,2)}$ and $M_{(1,2,1)}$ differ.

According to \cite[Corollary 2.8]{SW16}, every symmetric CQF is palindromic. When showing that a particular graph has a nonsymmetric CQF, it is sometimes feasible to show the stronger result that the CQF is nonpalindromic.

On the other hand, showing that a given CQF is palindromic appears to be a difficult task. Below, we present one criterion that guarantees a graph has a palindromic CQF. 

Recall that two labeled graphs $(G, L)$ and $(G', L')$ are \emph{isomorphic} if there exists a graph isomorphism $\omega\colon G\to G'$ such that $L'(\omega(v)) = L(v)$ for each vertex $v\in V(G)$. In this case, we write $(G, L)\cong (G', L')$.

\begin{definition} \label{def:flip_map} Let $(G, L)$ be a labeled graph with vertex set $V=\{v_1, ..., v_n\}$ with the labeling map $L\colon V\to [n]$. The \emph{flip map} $f$ sends $(G, L)$ to the labeled graph $(G', L')$ by defining $G'$ to have the vertex set $V$ and the labeling $L'(v_i)=n+1-L(v_i)$. For brevity, we write $f(G)\colonequals G'$.  We say $G$ is \emph{invariant under the flip map} if $(G, L)\cong(G', L')$.
\end{definition}

We present three examples of trees that are invariant under the flip map.

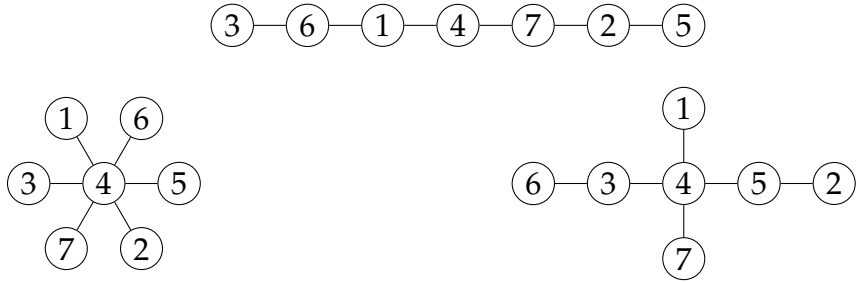
\begin{center}
\begin{tikzpicture}
    \foreach \i/\label in {1/3, 2/6, 3/1, 4/4, 5/7, 6/2, 7/5} {
        \node[draw, circle, inner sep=2pt] (v1\i) at (\i,0) {\label};
    }
    
    \foreach \i in {1,...,6} {
        \draw (v1\i) -- (v1\the\numexpr\i+1\relax);
    }
\end{tikzpicture}
\end{center}
\vskip .15in
\begin{center}
    \begin{minipage}{0.4\textwidth}
        \begin{tikzpicture}
            \node[draw, circle, inner sep=2pt] (v1) at (0,0) {4};

            \foreach \i/\label in {1/6, 2/1, 3/3, 4/7, 5/2, 6/5} {
                \node[draw, circle, inner sep=2pt] (v1\i) at (\i*60:1) {\label};
            }

            \foreach \i in {1,...,6} {
                \draw (v1) -- (v1\i);
            }
        \end{tikzpicture}
    \end{minipage}
    \begin{minipage}{0.32\textwidth}
        \begin{tikzpicture}
            \node[draw, circle, inner sep=2pt] (v1) at (0,0) {4};

            \node[draw, circle, inner sep=2pt] (v2) at (0,1) {1};
            \node[draw, circle, inner sep=2pt] (v3) at (0,-1) {7};

            \node[draw, circle, inner sep=2pt] (v4) at (1,0) {5};
            \node[draw, circle, inner sep=2pt] (v5) at (2,0) {2};

            \node[draw, circle, inner sep=2pt] (v6) at (-1,0) {3};
            \node[draw, circle, inner sep=2pt] (v7) at (-2,0) {6};

            \draw (v1) -- (v2);
            \draw (v1) -- (v3);
            \draw (v1) -- (v4);
            \draw (v4) -- (v5);
            \draw (v1) -- (v6);
            \draw (v6) -- (v7);
        \end{tikzpicture}
    \end{minipage}
\end{center}

\begin{proposition}
\label{prop:sum_Q_Q'}
Suppose $Q$ is the CQF of a labeled graph $G$, and $Q'$ is the CQF of the labeled graph $G'$, obtained from $G$ via the flip map.  Then $Q+Q'$ is palindromic.
\end{proposition}
\begin{proof}
Any proper coloring $c$ of $G$ is also a proper coloring of $G'$ (and vice-versa), corresponding to the same monomial $x^c$, as the flip map permutes the labels of the vertices of $G$, but does not change the color of each vertex. 

It is straightforward to check that $\text{Asc}_G(c)=\text{Des}_{G'}(c)$. Therefore, we also have that $\text{Asc}_{G'}(c)=\text{Des}_{G}(c)$, and so
\[
Q'=\sum_{c \in {\mathcal{PC}}(G')} x^c q^{asc_{G'}(c)}= \sum_{c \in {\mathcal{PC}}(G)} x^{c} q^{des_G(c)}=\rho(Q),
\]
where the last equality follows from \cite[Theorem 8.5.3 (b)]{Sag20}. Since $\rho$ is an involution, we then have:
\[
\rho(Q+Q')=\rho(Q+\rho(Q))=\rho(Q)+Q=Q'+Q.
\]
So, $Q+Q'$ is indeed palindromic.
\end{proof}

The proposition above allows us to conclude the following.
\begin{corollary}\label{cor:invariant-flip-map}
Suppose $Q$ is the CQF of a labeled graph $G$ and suppose $Q'$ is the CQF of the labeled graph $G'=f(G)$. If $G$ is isomorphic to $G'$ as labeled graphs, then $Q$ is palindromic.
\end{corollary}
\begin{proof}
    By Proposition \ref{prop:sum_Q_Q'} and the fact that $Q=Q'$:
    \[
    \rho(Q)=\rho\left(\frac{Q+Q'}{2}\right)=\frac{Q+Q'}{2}=Q,
    \]
    as desired.  \end{proof}

\section{Labeled Paths and Ribbon Tableaux}\label{sec:ribbon}


Given a labeled path graph $(P_n, L)$, its vertex set is $V=\{v_1, \ldots, v_n\}$ where  $v_i v_{i+1}$ is an edge for $1\leq i\leq n-1$. We can then define a permutation $\sigma = \sigma_1 \sigma_2 \cdots \sigma_n$ where $\sigma_i \colonequals L(v_i)$.

\begin{definition} Given a labeled path $P_n$ with associated permutation $\sigma_1\sigma_2{\cdots}\sigma_n$, the \emph{ascent-descent pattern}, or \emph{$ad$-pattern} for short , is a word $w_1w_2\cdots w_{n-1}$, in the alphabet $\{a,d\}$, with $w_i = a$ if $\sigma_i < \sigma_{i+1}$ and $w_i=d$ if $\sigma_i > \sigma_{i+1}$, for each $1\leq i\leq n-1$.
\end{definition}

For example, consider the two labelings of $P_5$ below.

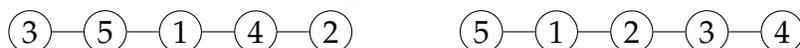
\begin{center}
\begin{tikzpicture}
    \foreach \i/\label in {1/3, 2/5, 3/1, 4/4, 5/2} {
        \node[draw, circle, inner sep=2pt] (v1\i) at (\i,0) {\label};
    }
    
    \foreach \i in {1,...,4} {
        \draw (v1\i) -- (v1\the\numexpr\i+1\relax);
    }

    \foreach \i/\label in {1/5, 2/1, 3/2, 4/3, 5/4} {
        \node[draw, circle, inner sep=2pt] (v2\i) at (\i+6,0) {\label};
    }
    
    \foreach \i in {1,...,4} {
        \draw (v2\i) -- (v2\the\numexpr\i+1\relax);
    }
    
\end{tikzpicture}
\end{center}
The $ad$-pattern for the left graph is $adad$, and for the right graph, $daaa$. 

By definition, the CQF of a path is invariant under permuting the labels but maintaining the order of the labels of adjacent vertices. To understand the CQF of a labeled path, it therefore suffices to focus on its $ad$-pattern alone.

We can visualize an $ad$-pattern using a \emph {ribbon diagram} (also called a \emph{rim-hook} or \emph{border strip}) with $n$ boxes as follows. We start with a single box and then sequentially append a new box to the right (resp. above) for each $a$ (resp. $d$) in the $ad$-pattern. For example, the two graphs above have ad-patterns
\begin{center}
\begin{tikzpicture}
    \foreach \i/\label in {1/3, 2/5, 3/1, 4/4, 5/2} {
        \node[draw, circle, inner sep=2pt] (v1\i) at (\i,0) {\label};
    }
    
    \foreach \i/\j/\lbl in {1/2/$a$, 2/3/$d$, 3/4/$a$, 4/5/$d$} {
        \draw (v1\i) -- (v1\j) node[midway, above] {\lbl};
    }

    \draw[->] (3,-0.5) -- ++(0,-0.4) node[below] {$adad$};

    \foreach \i/\label in {1/5, 2/1, 3/2, 4/3, 5/4} {
        \node[draw, circle, inner sep=2pt] (v2\i) at (\i+6,0) {\label};
    }
    
    \foreach \i/\j/\lbl in {1/2/$d$, 2/3/$a$, 3/4/$a$, 4/5/$a$} {
        \draw (v2\i) -- (v2\j) node[midway, above] {\lbl};
    }
    
    \draw[->] (9,-0.5) -- ++(0,-0.4) node[below] {$daaa$};
    
\end{tikzpicture}
\end{center}
and hence ribbon diagrams

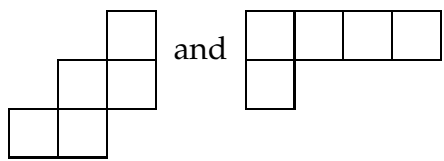
\[\begin{ytableau} \none & \none & \, \\ 
\none&\,&\,\\
\,&\,&\none\end{ytableau} \hskip .05in\text{ and }\hskip .05in
\begin{ytableau}\,&\,&\,&\,\\
\,&\none&\none&\none\\\end{ytableau}\]
respectively.

Note that we can completely describe a ribbon diagram with $n$ boxes by a composition $\alpha=(\alpha_1,\alpha_2, \ldots, \alpha_\ell)\vDash n$, where $\alpha_i$ is the number of boxes in row $i$ of the diagram, with row $1$ being the bottom row. For example, the two ribbon diagrams above have corresponding compositions of $(2,2,1)$ and $(1,4)$. We call the ribbon diagram whose composition is $\alpha$ the \emph{$\alpha$ ribbon}. For example, the left ribbon diagram above is the $(2,2,1)$ ribbon. We will also refer to the CQF of a labeled path as the CQF of its ribbon diagram.

\begin{definition}
    Any contiguous collection of boxes of a ribbon diagram $D$ is called a \emph{sub-ribbon} of $D$, or more specifically, a $\beta$ sub-ribbon, if it has composition $\beta$. 
\end{definition}
For example, the shaded boxes form a $(1,2,2)$ sub-ribbon of the ribbon diagram:
$$
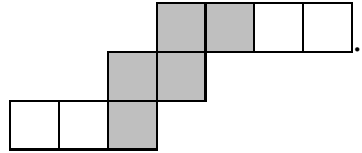
\begin{ytableau}
    \none & \none & \none & \none & *(lightgray) & *(lightgray) &\,&\,\\
    \none & \none & \none & *(lightgray) & *(lightgray)& \none &\none &\none\\
    \none &\,&\,& *(lightgray)&\none &\none &\none &\none
\end{ytableau}.$$

We can use a ribbon tableau to encode a coloring of a labeled path $P_n$. Given a proper coloring $c\in\mathcal{PC}(P_n)$ of a labeled path $P_n$ with vertex set $[n]$, we place the color $c(i)$ in the $i$th box of its ribbon diagram, starting with the bottom, left box, and following the ribbon, creating its \emph{ribbon tableau}. For example, below, we have two colored paths (with the color of each vertex indicated above it) and their corresponding ribbon tableaux.

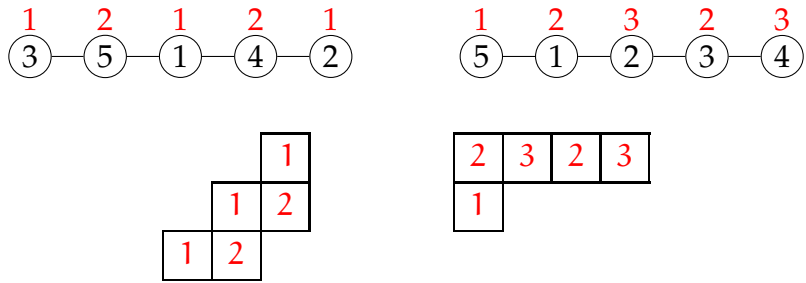
\begin{center}
\begin{tikzpicture}
    \foreach \i/\label in {1/3, 2/5, 3/1, 4/4, 5/2} {
        \node[draw, circle, inner sep=2pt] (v1\i) at (\i,0) {\label};
    }
    
    \foreach \i in {1,...,4} {
        \draw (v1\i) -- (v1\the\numexpr\i+1\relax);
    }

    \foreach \i/\label in {1/1, 2/2, 3/1, 4/2, 5/1} {
        \node[above=2mm, red] at (v1\i) {\label};
    }

    \foreach \i/\label in {1/5, 2/1, 3/2, 4/3, 5/4} {
        \node[draw, circle, inner sep=2pt] (v2\i) at (\i+6,0) {\label};
    }

    \foreach \i in {1,...,4} {
        \draw (v2\i) -- (v2\the\numexpr\i+1\relax);
    }

     \foreach \i/\label in {1/1, 2/2, 3/3, 4/2, 5/3} {
        \node[above=2mm, red] at (v2\i) {\label};
    }
    
\end{tikzpicture}
\end{center}
\vskip .2in
\[\begin{ytableau} \none & \none & \,\color{red} 1 \\ 
\none&\color{red} 1&\color{red} 2\\
\color{red} 1&\color{red} 2&\none\end{ytableau} \hskip .75in
\begin{ytableau}\color{red} 2&\color{red} 3&\color{red} 2&\color{red} 3\\
\color{red} 1&\none&\none&\none\\\end{ytableau}\]
We define $RD(P_n)$ to be the ribbon diagram associated to a (labeled) path $P_n$. We define the \emph{proper coloring of $RD(P_n)$}, denoted $RT(c)$, to be the ribbon tableau associated to a proper coloring $c\in \mathcal{PC}(P_n)$ with the number in each box called its \emph{color}.

In Section \ref{sec:path}, certain boxes of $RD(P_n)$ will have special significance, so we define them now.
\begin{definition} Given a path $P_n$, whose vertices are labeled with $[n]$, and corresponding ribbon diagram $RD(P_n)$, a \emph{left-upper} ($LU$) \emph {corner} of $RD(P_n)$ is a box with no box to its left and no box above it.  A \emph{right-lower} ($RL$) \emph{corner} is a box with no box to its right and no box below it. 
\end{definition}

For example, in the two ribbon diagrams below, the left-upper corners are labeled ``$LU$,'' and the right-lower corners are labeled ``$RL$.''

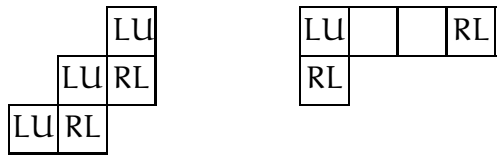
\[\begin{ytableau} \none & \none & \,LU \\ 
\none&LU&RL\\
LU&RL&\none\end{ytableau} \hskip .75in
\begin{ytableau}LU&\,&\,&RL\\
RL&\none&\none&\none\\\end{ytableau}\]

Observe that the maximum ascent number of a proper coloring $c$ of a graph $G=(V, E)$ equals $|E|$. This maximum is achieved if and only if, for each edge $ij\in E$, we have $c(i)< c(j)$ whenever $i<j$. If $G=P_n$, this happens if and only if the colors of $RT(c)$ \textit{increase left to right in each row} and \textit{increase from top to bottom in each column}. In particular, the color 1 must always be assigned to an $LU$ box.

More generally, if $P_n$ does not have the natural labeling, for a given coloring $c$ and its corresponding colored ribbon tableau $RT(c)$, $\operatorname{asc}(c)$ equals the number of adjacent horizontal pairs in $RT(c)$ whose colors increase left to right plus the number of adjacent vertical pairs whose colors increase top to bottom.

\section{The chromatic quasisymmetric function of the path graph}\label{sec:path}

For any labeled graph $G$ with $n$ vertices, consider the expansion of $X(G; \mathbf{x}, q)$ in the monomial quasisymmetric function basis:
$$
X(G; \mathbf{x}, q) = \sum_{\alpha \ \vDash \ n} 
{c}_{\alpha}(q) M_{\alpha}.
$$
Note that the coefficients $c_{\alpha}(q)$ are polynomials in $q$. If $G=P_n$, then $M_\alpha$ appears in $X(P_n; \mathbf{x}, q)$ if and only if there is a proper coloring of $RD(P_n)$ that uses color $i$ $\alpha_i$ times, for $i=1,2,\dots$. We say such a coloring has \emph {content} $\alpha$. 
The coefficient of $q^k$ in $c_\alpha(q)$ equals the number of proper colorings of $P_n$ using content $\alpha$, with $k$ ascents.

For example, if $P_4$ is
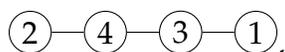
\begin{center}
\begin{tikzpicture}
    \foreach \i/\label in {1/2, 2/4, 3/3, 4/1} {
        \node[draw, circle, inner sep=2pt] (v1\i) at (\i,0) {\label};
    }
    
    \foreach \i in {1,...,3} {
        \draw (v1\i) -- (v1\the\numexpr\i+1\relax);
    }
\end{tikzpicture},
\end{center}
then $RD(P_4)$ is
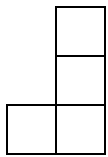
\[
\begin{ytableau} \none &\,\\
\none&\,\\
\,&\,\end{ytableau}
\]
and
\begin{equation*} 
\begin{split}
X(P_4, \mathbf{x},q) & = (3q^3+9q^2+9q+3)M_{(1, 1, 1, 1)}\\
& + (2q^2+3q+1)M_{(1, 1, 2)} + (q^3+2q^2+2q+1)M_{(1, 2, 1)}\\
& + (q^3+3q^2+2q)M_{(2, 1, 1)} + (q^2+q)M_{(2, 2)}.
\end{split}
\end{equation*}
We see that $c_{(1,1,2)}(q) = 2q^2+3q+1$, so there are two colorings of $RD(P_4)$ that yield $2$ ascents and use color $1$ one time, color $2$ one time, and color $3$ two times, as shown below.
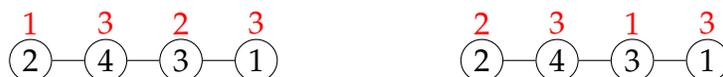
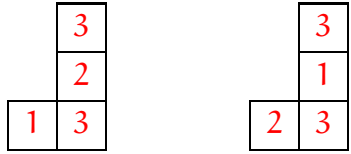
\begin{center}
\begin{tikzpicture}
    \foreach \i/\label in {1/2, 2/4, 3/3, 4/1} {
        \node[draw, circle, inner sep=2pt] (v1\i) at (\i,0) {\label};
    }
    
    \foreach \i in {1,...,3} {
        \draw (v1\i) -- (v1\the\numexpr\i+1\relax);
    }

    \foreach \i/\label in {1/1, 2/3, 3/2, 4/3} {
        \node[above=2mm, red] at (v1\i) {\label};
    }

    \foreach \i/\label in {1/2, 2/4, 3/3, 4/1} {
        \node[draw, circle, inner sep=2pt] (v2\i) at (\i+6,0) {\label};
    }

    \foreach \i in {1,...,3} {
        \draw (v2\i) -- (v2\the\numexpr\i+1\relax);
    }

     \foreach \i/\label in {1/2, 2/3, 3/1, 4/3} {
        \node[above=2mm, red] at (v2\i) {\label};
    }
    
\end{tikzpicture}
\end{center}
\vskip .15in
\[
\begin{ytableau} \none &\color{red}3\\
\none&\color{red}2\\
\color{red}1&\color{red}3\end{ytableau}
\hskip .75in
\begin{ytableau} \none &\color{red}3\\
\none&\color{red}1\\
\color{red}2&\color{red}3\end{ytableau}
\]
    
In what follows, we assume that the vertices of $P_n$ are labeled with $[n]$, but not in the natural order.  Hence its $ad$-pattern is not $aa\cdots a$ or $dd \cdots d$, and so $RD(P_n)$ does not have composition $(n)$ or $(1^n)$.

We prove Theorem~\ref{thm:main} by analyzing several types of configurations or sub-ribbons that $RD(P_n)$ might contain. First, we can reduce the number of cases with the following:
\begin{proposition}\label{prop:flip}
    If a labeled path $P_n$ has a symmetric CQF, then the labeled path obtained by applying the flip map of Definition~\ref{def:flip_map} has a symmetric CQF.
\end{proposition}
\begin{proof}
    
     We extend the definition of the flip map to a colored path graph, $P_n$. Let $f(P_n)$ be the colored, labeled path obtained from $P_n$ by replacing label $i$ with $n+1-i$, but retaining the coloring. Then 
    $$
    X(f(P_n); \mathbf{x}, q) = \sum_{c\in\mathcal{PC}(f(P_n))} \mathbf{x}^{c} q^{\operatorname{asc}(c)}
    = \sum_{c\in\mathcal{PC}(P_n)} \mathbf{x}^{c} q^{\operatorname{des}(c)}.
    $$
    By \cite[Corollary 2.7]{SW16}, since $X(P_n, \mathbf{x},q)$ is symmetric, we have that
    $$
    X(P_n, \mathbf{x},q)=\sum_{c\in\mathcal{PC}(P_n)} \mathbf{x}^{c} q^{\operatorname{des}(c)}.
    $$
    Thus, 
    \[
X(f(P_n); \mathbf{x}, q) = X(P_n; \mathbf{x}, q)
    \]
    and both are symmetric.
\end{proof}
Translating this result to ribbon diagrams, we note that the $ad$-pattern of $f(P_n)$ is obtained from the $ad$-pattern of $P_n$ by replacing each `$a$' with a `$d$' and each `$d$' with an `$a$'.  This corresponds to reflecting the ribbon diagram $RD(P_n)$ across the diagonal that starts at the lower left corner of the ribbon diagram and has slope $1$, resulting in the ``reflection'' of the ribbon diagram.

For example, if $P_n$ is 
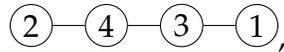
\begin{center}
\begin{tikzpicture}
 \foreach \i/\label in {1/2, 2/4, 3/3, 4/1} {
        \node[draw, circle, inner sep=2pt] (v1\i) at (\i,0) {\label};
    }
    
    \foreach \i in {1,...,3} {
        \draw (v1\i) -- (v1\the\numexpr\i+1\relax);
    }
    \end{tikzpicture},
    \end{center}
    then $f(P_n)$ is obtained from $P_n$ by replacing vertex label $i$ with $4-i+1$:
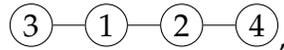
\begin{center}
\begin{tikzpicture}
 \foreach \i/\label in {1/3, 2/1, 3/2, 4/4} {
        \node[draw, circle, inner sep=2pt] (v1\i) at (\i,0) {\label};
    }
    
    \foreach \i in {1,...,3} {
        \draw (v1\i) -- (v1\the\numexpr\i+1\relax);
    }
    \end{tikzpicture},
    \end{center}
    and
    their corresponding ribbon diagrams are
    
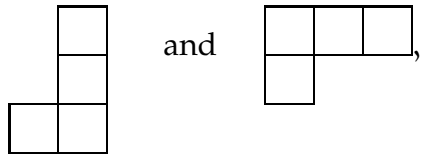
\[
\begin{ytableau} \none &\,\\
\none&\,\\
\,&\,\end{ytableau}\
\hskip .25in \text{and}\hskip .25in
\begin{ytableau} \,&\,&\,\\
\,&\none&\none\end{ytableau},
\]
respectively, with the second ribbon diagram being the reflection of the first. 

Proposition \ref{prop:flip} allows us to conclude that any statement about the symmetry, or lack thereof, of the CQF of a ribbon diagram immediately applies to the CQF of its reflection.\\

The proposition below reduces our study to ribbon diagrams with the same number of $LU$ and $RL$ corners. Namely, it shows that the corresponding labeled paths do not have symmetric CQFs if these counts are different.  Two examples of such ribbon diagrams are: 
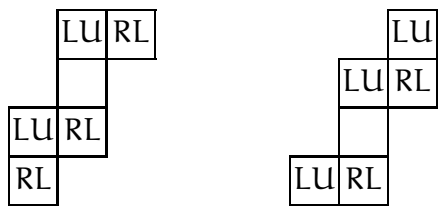
\[ \begin{ytableau} \none&LU&RL\\
\none&\,&\none\\
LU &RL&\none\\
RL&\none&\none\end{ytableau}\hskip .7in
\begin{ytableau} \none&\none&LU\\
\none&LU&RL\\
\none&\,&\none\\
LU&RL&\none&\none\\\end{ytableau}
\]

\begin{proposition}\label{prop:LURL}
    Let $P_n$ be a labeled graph such that the number of LU and RL corners of $RD(P_n)$ are different. Then, $X(P_n; \mathbf{x}, q)$ is not palindromic, and hence not symmetric.
\end{proposition}

\begin{proof} Let $k$ be the number of $LU$ corners and $j$ be the number of $RL$ corners.
    
If $k<j$, then $q^{|E|}M_{(k, 1^{n-k-j}, j)}$ is a term of $X(P_n; \mathbf{x}, q)$. This is because we can construct a coloring $c$ with content $(k, 1^{n-k-j}, j)$ and $\operatorname{asc}(c)=|E|$ as follows. The $LU$ corners can be colored $1$, the $RL$ corners can be colored with the largest color, and all other boxes can be colored using the remaining colors so that row and column entries increase. However, $q^{|E|}M_{(j, 1^{n-k-j}, k)}$ is \emph{not} a term of $X(P_n; \mathbf{x}, q)$ because there is no way to place color $1$ in $j>k$ boxes, while maximizing ascents, since, to maximize ascents, color $1$ can only be placed in the $k$ $LU$ corners.

If $k > j$, a similar argument shows that $q^{|E|}M_{(k, 1^{n-k-j}, j)}$ is a term of $X(P_n; \mathbf{x}, q)$, but $q^{|E|}M_{(j, 1^{n-k-j}, k)}$ is not. Note that the largest color can only be placed in the $j$ $RL$ corners for maximum ascent.

In either case, we have found a composition $\alpha$ such that the coefficients of $M_{\alpha}$ and $M_{\alpha^r}$ are not equal, and hence $X(P_n; \mathbf{x}, q)$ is not palindromic.
\end{proof}

Proposition \ref{prop:LURL} was proved independently and recast in terms of acyclic directed graphs, sources and sinks in \cite[Lemma 4.2]{GPS24}. Note that our LU corners and RL corners correspond to their sources and sinks, respectively. 

We say that $i$ consecutive rows (resp. columns) that each contain at least 2 boxes is a \emph{stack} of $i$ rows (resp. columns). If all rows (resp. columns) contain at least two boxes, we say the ribbon diagram consists of \emph{stacked} rows (resp. columns). A ribbon diagram of stacked rows is shown below.

$$\begin{ytableau}\none&\none&\none&\none&\none&\,&\,&\,&\,\\
    \none&\none&\none&\,&\,&\,&\none&\none&\none\\
    \none&\,&\,&\,&\none&\none&\none&\none&\none\\
    \,&\,&\none&\none&\none&\none&\none&\none&\none\\
    \end{ytableau}$$

Our next result shows that labeled paths whose ribbon diagrams consist of stacked rows do not have a symmetric CQF.  

\begin{notation} If $f(x)$ is a formal power series, we denote the coefficient of $x^n$ in $f(x)$ by $[x^n]f(x).$\label{coeff_notation}
\end{notation}

\begin{proposition}\label{prop:stack-rows}
    Suppose that $P_n$ is a labeled path whose composition corresponding to $RD(P_n)$ is $\alpha = (\alpha_1,\alpha_2, \ldots,\alpha_k)$ with $\alpha_i \geq 2$ for each $i$ and $k \geq 2$. Then, $X(P_n; \mathbf{x}, q)$ is not symmetric. 
\end{proposition}

\begin{proof}

Let $r,s$ be the lengths of any two adjacent rows of $RD(P_n)$, i.e. $(r,s)=(\alpha_i,\alpha_{i+1})$. Choose $b$ to be any integer in the interval $[r,n-k-s+2]$. Note this interval is non-empty since $(r-1)+(s-1)+k \leq n$, where the left-hand side counts the number of boxes in rows $i,i+1$, as well as the $LU$ corners in $RD(P_n)$. 

To prove that $X(P_n; \mathbf{x}, q)$ is not symmetric, we show that 
\[
[q^{|E|}M_{(k,1^{n-k})}] X (P_n ; \mathbf{x},q) > [q^{|E|}M_{(1^{b-1},k,1^{n-k-b+1})}] X (P_n ; \mathbf{x},q).
\]
Let $A$ denote the set of proper colorings of $RD(P_n)$ with content $(k, 1^{n-k})$ and maximum ascent number $|E|$. Each coloring in $A$ contributes  $q^{|E|}M_{(k,1^{n-k})}$ to $X(P_n; \mathbf{x}, q)$. Note that these colorings use color $1$ exactly $k$ times and every other color once. The set $A$ is not empty because every row has at least two boxes. Similarly, let $B$ denote the set of proper colorings of $RD(P_n)$ with content $(1^{b-1},k,1^{n-k-b+1})$ and maximum ascent number $|E|$. Each coloring in $B$ contributes $q^{|E|}M_{(1^{b-1},k,1^{n-k-b+1})}$ to $X(P_n; \mathbf{x}, q)$. Note that these colorings use color $b$ exactly $k$ times and every other color once.

We extend $B$ to $B'$, the set of (not necessarily proper) colorings of $RD(P_n)$ with content $(1^{b-1},k,1^{n-k-b+1})$ such that the rows are strictly increasing 
left to right, while the columns are weakly increasing top to bottom.

For example, consider the $(3,3,4)$ ribbon diagram.  Let $i=1$, so $r=3$ and $s=3$. Then, since $k=3$, we have $b\in[3, 6]$; we choose $b=4$.  With content $(1^3, 3, 1^4)$,
the ribbon tableau below, constructed according to the proof of this proposition, is an element of $B'$, but the coloring is not proper, so it is not an element of $B$.

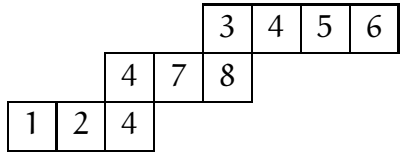
$$\begin{ytableau}\none&\none&\none&\none&3&4&5&6\\
    \none&\none&4&7&8&\none&\none&\none\\
    1&2&4&\none&\none&\none&\none&\none\\
    \end{ytableau}$$

As noted previously, because every coloring $c \in A$ has a maximum ascent number, the colors of $c$ must increase in each row. Since there are $k$ rows and $k$ copies of color 1, the leftmost box of each row must be colored $1$.

Define a bijection $\zeta\colon A \to B'$ as follows: Given $c \in A$, identify the unique $k-1$ rows not containing $b$. Replace the $k-1$ occurrences of $1$'s in these rows with $b$'s and sort each row so that its colors are increasing left to right.
 For example, if
 \ytableausetup{smalltableaux}
 $$c_1=\begin{ytableau}\none&\none&\none&\none&1&3&4&8\\
    \none&\none&1&6&7&\none&\none&\none\\
    1&2&5&\none&\none&\none&\none&\none\\
    \end{ytableau},
    \hspace{0.75cm}
    \text{ then }
    \zeta(c_1)=\begin{ytableau}\none&\none&\none&\none&1&3&4&8\\
    \none&\none&4&6&7&\none&\none&\none\\
    2&4&5&\none&\none&\none&\none&\none\\
    \end{ytableau},$$
 which is in $B$.  But if
$$c_2=\begin{ytableau}\none&\none&\none&\none&1&3&7&8\\
    \none&\none&1&5&6&\none&\none&\none\\
    1&2&4&\none&\none&\none&\none&\none\\
    \end{ytableau},
    \hspace{0.75cm}
    \text{ then }
\zeta(c_2)=\begin{ytableau}\none&\none&\none&\none&3&4&7&8\\
    \none&\none&4&5&6&\none&\none&\none\\
    1&2&4&\none&\none&\none&\none&\none\\
    \end{ytableau},$$
    which is in $B'\setminus B$.
 \ytableausetup{nosmalltableaux} 
 
By construction, $\zeta(c) \in B'$ since $\zeta(c)$ could only (potentially) have a descent in a column consisting of the box at the end of a row and the box immediately above it.
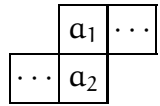
\begin{align*}
\begin{ytableau}
    \none&a_1&\cdots\\
    \cdots&a_2
\end{ytableau}
\end{align*}
But every row of $\zeta(c)$ must contain color $b$ exactly once, and the rows are increasing, so $a_1 \leq b$ and $b \leq a_2$. Thus, the columns of $\zeta(c)$ weakly  increase from top to bottom.
 
By definition, each coloring $c' \in B'$ has exactly one $b$ per row, and there will be exactly one row that contains both $1$ and $b$. With this observation, 
we define the inverse $\zeta^{-1}: B' \to A$ as follows: Given $c' \in B'$, identify the $k-1$ rows not containing $1$, replace each occurrence of $b$ in these rows with $1$, and sort each row so its colors increase left to right. 

Next, we show that there exists a coloring in $B' \setminus B$, from which it follows that $$[q^{|E|}M_{(k,1^{n-k})}] X (P_n ; \mathbf{x},q) = |A| = |B'| > |B| = [q^{|E|} M_{(1^{b-1},k,1^{n-k-b+1})} ] X (P_n ; \mathbf{x},q). $$

Let  $m = n-k+1$ be the maximum color used by colorings in $B'$. We construct a coloring in $B'\setminus B$ as follows.
\begin{itemize}
\item Fill row $i$ of $RD(P_n)$ with $(1,2,\ldots,r-1,b)$ and note that the row is increasing left to right since $b \geq r$.
\item Fill row $i+1$ of $RD(P_n)$ with $(b,m-s+2, \ldots,m-1,m)$ and note that the row is increasing left to right since $b \leq n-k-s+2 = m-s+1$. 
\item 
Fill the rest of $RD(P_n)$ arbitrarily such that the coloring is in $B'$. 
\end{itemize}
Then the coloring is not in $B$ since the column (consisting of two boxes) joining rows $i$ and $i+1$ contains two adjacent copies of color $b$. \end{proof}

\begin{corollary}\label{cor:stack-columns}
    Suppose that $P_n$ is a labeled path with associated ribbon diagram consisting of at least two stacked columns. Then, $X(P_n; \mathbf{x}, q)$ is not symmetric. 
\end{corollary}

\begin{proof} The result follows by combining Proposition~\ref{prop:stack-rows} and Proposition~\ref{prop:flip}.
\end{proof}

The following two propositions show that ribbon diagrams containing certain sub-ribbons correspond to labeled paths that are not symmetric.

\begin{proposition}\label{prop:113-subribbon}
\ytableausetup{smalltableaux}
    Suppose that $P_n$ is a labeled path with vertex set $[n]$. If $RD(P_n)$ contains a $(1,1,3)$ sub-ribbon, $\begin{ytableau}
\,&\,&\,\\\,&\none&\none\\\,&\none&\none\end{ytableau},$ begins (on the lower left) with the $(1,3)$ sub-ribbon, $\begin{ytableau} \,&\,&\,\\\,&\none&\none\end{ytableau},$ or ends with the $(1,1,2)$ sub-ribbon, $\begin{ytableau}
    \,&\,\\\,&\none\\\,&\none
\end{ytableau},$
then $Q=X(P_n; \mathbf{x},q)$ is not symmetric.
\end{proposition}

\begin{proof}
Let $k$ be the number of $LU$ corners of $RD(P_n)$. We show that
\[
[q^{|E|} M_{(k+1,1^{n-k-1})}]Q = 0 \quad \text{but} \quad [q^{|E|} M_{(1,k+1,1^{n-k-2})}]Q > 0.
\]
For a proper coloring with $|E|$ ascents and content $(k+1,1^{n-k-1})$, each copy of color 1 must be placed in an LU corner to respect the increasing conditions on the rows and columns. There are $k < k+1$ LU corners, so such a coloring does not exist, hence $[q^{|E|} M_{(k+1,1^{n-k-1})}]Q = 0$. 

On the other hand, let $T$ be one of the sub-ribbons in the statement. Color the $LU$ corner of $T$ with color $1$ and the two neighboring boxes with color $2$. Color the remaining $k-1$ $LU$ corners of $RD(P_n)$ with $2$, and arbitrarily color the rest of $RD(P_n)$ with the remaining distinct colors $\{3,4,\ldots, n-k\}$ such that the rows are increasing left to right and columns are increasing top to bottom. This is a proper coloring with a maximum ascent number by construction. Thus, $[q^{|E|} M_{(1,k+1,1^{n-k-2})}]Q > 0$, giving the desired conclusion.
\end{proof}

To illustrate the proof of Proposition~\ref{prop:113-subribbon}, consider the following ribbon tableau with $k=3$ $LU$ corners. The given coloring, with the $(1,1,3)$ sub-ribbon shaded, shows that $[M_{(1,4,1^7)}]Q>0$.
\ytableausetup{nosmalltableaux}
\[
\begin{ytableau}
\none&\none&\none&\none&\none&2&9\\
\none&\none&\none&\none&\none&5&\none\\
\none&\none&*(lightgray)1&*(lightgray)2&*(lightgray)4&6&\none\\
\none&\none&*(lightgray)2&\none&\none&\none&\none\\
2&3&*(lightgray)8&\none&\none&\none&\none\\
7&\none&\none&\none&\none&\none&\none\end{ytableau}
\]


\begin{corollary}\label{cor:311-subribbon}
\ytableausetup{smalltableaux}
 Suppose that $P_n$ is a labeled path with vertex set $[n]$. If $RD(P_n)$ contains a $(3,1,1)$ sub-ribbon,  $\begin{ytableau}
\none, &\none &\,\\\none&\none & \, \\\,& & \end{ytableau},$ or begins with a $(2,1,1)$ sub-ribbon, $\begin{ytableau}
\none &\,\\ \none & \, \\ & \end{ytableau},$ or ends with a $(3,1)$ sub-ribbon, $\begin{ytableau} \none&\none &\, \\ \,&\,&\,\end{ytableau},$  has a nonsymmetric CQF.
\end{corollary}

\ytableausetup{nosmalltableaux}

\begin{proof}
    The result follows by combining Proposition~\ref{prop:113-subribbon} and Proposition~\ref{prop:flip}.
\end{proof}

\begin{definition} Suppose that $P_n$ is a labeled path. We say $RD(P_n)$ is \emph{regular} if $RD(P_n)$ has a row of length $2$ followed by a row of length at least $2$, or a terminal row of length $1$. In this case, the $(2,1)$ sub-ribbon contained in these rows is called the \emph{regular} $(2,1)$ \emph {sub-ribbon}. 
\end{definition}

For example, the ribbon diagram below is regular, with two examples of regular $(2,1)$ sub-ribbons shaded.
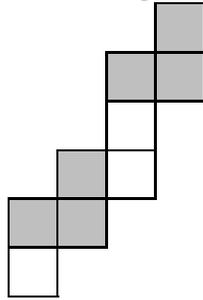
\begin{center}
\begin{ytableau}
\none&\none&\none&*(lightgray)\\
\none&\none&*(lightgray)&*(lightgray)\\
\none&\none&&\none\\
\none&*(lightgray)&&\none\\
*(lightgray)&*(lightgray)&\none&\none\\
&\none&\none&\none\end{ytableau}
\end{center}

\begin{proposition}\label{prop:regular}
    Suppose that $P_n$ is a labeled path such that
    \begin{itemize} \ytableausetup{smalltableaux}
        \item $RD(P_n)$ does not contain a $(1,1,3)$ sub-ribbon, $\begin{ytableau}
\,&\,&\,\\\,&\none&\none\\\,&\none&\none\end{ytableau}$, does not begin with the $(1,3)$ sub-ribbon, $\begin{ytableau} \,&\,&\,\\\,&\none&\none\end{ytableau}$, does not end with the $(1,1,2)$ sub-ribbon, $\begin{ytableau}
    \,&\,\\\,&\none\\\,&\none
\end{ytableau}$, and 
        \item $RD(P_n)$ is regular.
    \end{itemize}
    Then, $Q=X(P_n; \mathbf{x},q)$ is not symmetric. 
\end{proposition}


\begin{proof}
Let $k$ be the number of $LU$ corners of $RD(P_n)$. We show that
\[
[q^{|E|} M_{(k,1^{n-k})}] Q > [q^{|E|} M_{(1,k,1^{n-k-1})}] Q.
\]
Let $A$ be the set of proper colorings of $RD(P_n)$ with content $(k, 1^{n-k})$ and maximum ascent number $|E|$, whose cardinality equals the coefficient $[q^{|E|} M_{(k,1^{n-k})}]Q$. These colorings use $k$ copies of color $1$ and each of the next $n-k$ colors once. Let $B$ denote the set of proper colorings of $RD(P_n)$ with content $(1, k, 1^{n-k-1})$ and maximum ascent number $|E|$, whose cardinality equals the coefficient $[q^{|E|} M_{(1,k,1^{n-k-1})}]Q$. These colorings use $k$ copies of color $2$ and each of the colors $\{1, 3, 4, ..., n-k-1\}$ once.

To show that $|B|<|A|$, we produce an injective map $\psi\colon B\to A$ that is \emph{not} surjective. Define $\psi\colon B \to A$ as follows: given $T\in B$, $\psi(T)$ is obtained by replacing all $2$'s in the $LU$ corners of $T$ by $1$'s. 
\ytableausetup{smalltableaux}
For example, if 
$$T=
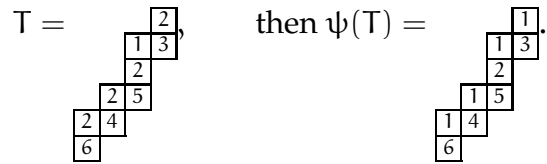
\begin{ytableau}
\none&\none&\none&2\\
\none&\none&1&3\\
\none&\none&2&\none\\
\none&2&5&\none\\
2&4&\none&\none\\
6&\none&\none&\none\end{ytableau},
\hspace{0.75cm}
\text{   then } 
\psi(T) = \begin{ytableau}\none&\none&\none&1\\
\none&\none&1&3\\
\none&\none&2&\none\\
\none&1&5&\none\\
1&4&\none&\none\\
6&\none&\none&\none\end{ytableau}.
$$

By the first hypothesis, each $T\in B$ contains $k-1$ $LU$ corners colored $2$ and one $LU$ corner colored $1$. As a result, $\psi$ is well-defined. Indeed, $\psi(T)$ has the maximum ascent number because replacing $2$s with $1$s maintains all ascents. Furthermore, $\psi$ is an injection since the \textit{unique $2$ in $\psi(T)$ is adjacent to exactly one $1$}, which uniquely determines $T$.

However, $\psi$ is not a surjection. By the second hypothesis, $RD(P_n)$ contains a regular $(2,1)$ sub-ribbon $S$. We construct a ribbon tableau $T'\in A$ not in the image of $\psi$: color the $LU$ corners of $S$ with color $1$, and its $RL$ corner with color $2$. Color the remaining $LU$ corners with color $1$ and the rest of $T'$ arbitrarily so that row and column entries strictly increase. Then $T'$ is \emph{not} in the image of $\psi$, since no coloring in $\psi(B)$ \textit{contains a $2$ that is adjacent to two $1$'s} because the purported pre-image would not be a proper coloring, as it would have two adjacent boxes colored $2$.
\end{proof}

Below is an example of such a ribbon tableau $T'$, with $S$ shaded.
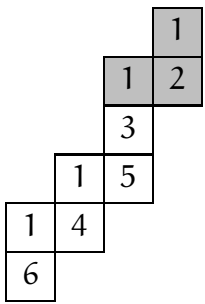
\ytableausetup{nosmalltableaux}
\[\begin{ytableau}
\none&\none&\none&*(lightgray)1\\
\none&\none&*(lightgray)1&*(lightgray)2\\
\none&\none&3&\none\\
\none&1&5&\none\\
1&4&\none&\none\\
6&\none&\none&\none\end{ytableau}\]

We now give a proof of the main theorem.

\begin{proof}[Proof of Theorem~\ref{thm:main}] Since $P_n$ with the natural labeling has a symmetric CQF \cite{SW16}, we only need to prove the converse. 

Let $P_n$ be a path whose labeling is not natural. By Proposition~\ref{prop:LURL}, we can assume that $P_n$ has the same number of $LU$ and $RL$ corners. Accordingly, $RD(P_n)$ has at least a stack of two rows or a stack of two columns. If $RD(P_n)$ consists entirely of a stack of rows or a stack of columns, $X(P_n; \mathbf{x}, q)$ is not symmetric by Proposition~\ref{prop:stack-rows} and Corollary~\ref{cor:stack-columns}. Assume $RD(P_n)$ has a mix of stacks of rows and stacks of columns. By Proposition~\ref{prop:flip}, we may assume that $RD(P_n)$ starts with a stack of columns (including, possibly, columns with two boxes). Consider the first transition point from the stack of columns to the stack of rows, where we must necessarily encounter a row $\mathbf{r}$ of length $\geq 3$, colored in olive below. 
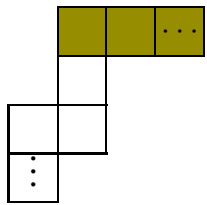
\[
\hskip .7in
\begin{ytableau}
\none &\none &*(olive) &*(olive) &*(olive)\cdots\\
\none &\none & \\
\none & & \\
\none &\vdots &\none \end{ytableau}
\]
If the column left adjacent to $\mathbf{r}$ has at least $3$ boxes (colored in cyan below), then $RD(P_n)$ has a $(1,1,3)$ sub-ribbon and $X(P_n; \mathbf{x}, q)$ is not symmetric by Proposition~\ref{prop:113-subribbon}. 
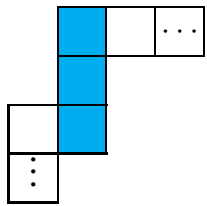
\[
\hskip .7in
\begin{ytableau}
\none &\none &*(cyan) & &\cdots\\
\none &\none &*(cyan) \\
\none & &*(cyan) \\
\none &\vdots &\none \end{ytableau}
\]
Thus, we may assume that the column $\mathbf{c}$ left adjacent to $\mathbf{r}$ has $2$ boxes. If $\mathbf{c}$ is the first column of $RD(P_n)$, then the tableau begins with:
$$
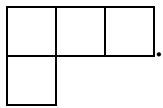
\begin{ytableau} \,&\,&\,\\
\,&\none&\none\end{ytableau}.
$$
In this case, we can apply Proposition~\ref{prop:113-subribbon} to deduce that $X(P_n; \mathbf{x}, q)$ is not symmetric. Otherwise, there is at least one more column to the left of $\mathbf{c}$. In this juncture, we have a regular $(2, 1)$ sub-ribbon: 
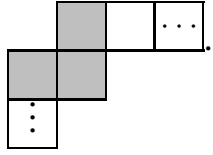
\[
\hskip .7in
\begin{ytableau}
\none&*(lightgray)& & \cdots\\
*(lightgray)&*(lightgray) \\
\vdots\end{ytableau}.
\]
Now, either $RD(P_n)$ contains a $(1,1,3)$ sub-ribbon, begins with a $(1,3)$ sub-ribbon, or ends with a $(1,1,2)$ sub-ribbon, so we are done by Proposition \ref{prop:113-subribbon}, or it does not, so we are done by Proposition~\ref{prop:regular}. Hence, we conclude in all possible cases that  $X(P_n; \mathbf{x}, q)$ is not symmetric. 
\end{proof}

\section{The chromatic quasisymmetric function of other trees}\label{sec:other-trees}

Theorem \ref{thm:main} shows that only the natural labeling of the path graph $P_n$ results in a symmetric CQF.  However, it is not the case that every tree will have some labeling that results in a symmetric CQF, as shown by the following theorems.

Consider the \emph{star graph} $K_{1, n-1}$ on $n\geq 3$ vertices. The \emph{central vertex} is the unique vertex of degree $>1$.

\begin{proposition}\label{prop:star-symmetry}
The CQF of a (labeled) star graph $G$ with $n$ vertices is palindromic if and only if $n$ is odd and the central vertex is labeled $\frac{n+1}{2}$.
\end{proposition}

\begin{proof}
Suppose the central vertex of $G$ is labeled $j$ where $1\leq j\leq n$. Consider the coefficients of the two monomial quasisymmetric functions $M_{(1, n-1)}$ and $M_{(n-1, 1)}$ in $X(G; x, q)$. First note that they each consist of only one term, $q^{r}$ and $q^{s}$, respectively. This is because there is only one way to color the vertices with content $(1, n-1)$, namely, the central vertex receives color $1$ and other vertices receive color $2$; the ascent number for this coloring is $r=n-j$. Similarly, there is only one way to color the vertices with content $(n-1,1)$: the central vertex receives color $2$, and other vertices receive color $1$; the ascent number for this coloring is $s=j-1$. Note that $r+s=n-1$. 

If $n$ is even, then $r\neq s$, and $X(K_{1,n-1};\mathbf{x}, q)$ is not palindromic.  If $n$ is odd, and $r\neq s$, then $X(K_{1,n-1};\mathbf{x}, q)$ is not palindromic. Thus, $r=s=\frac{n-1}{2}$ is a necessary condition for $X(K_{1,n-1};\mathbf{x}, q)$ to be palindromic. Note that $r=s=\frac{n-1}{2}$ if and only if the central vertex is labeled $\frac{n+1}{2}$. In this case, the labeled graph is invariant under the flip map of Definition~\ref{def:flip_map} and thus,  the coefficient of $q^i$ is equal to that of $q^{n-1-i}$ and $X(K_{1,n-1};\mathbf{x}, q)$ is palindromic.
\end{proof}

\begin{proposition}
The CQF of a (labeled) star graph $G$ with $n\geq 4$ vertices is never symmetric.
\end{proposition}

\begin{proof}
Symmetric CQFs are palindromic by \cite[Corollary 2.8]{SW16}. Thus, by Proposition~\ref{prop:star-symmetry}, we need only show that the star graph on $n$ vertices (for $n$ odd) and central vertex labeled $\frac{n+1}{2}$ does not have a symmetric CQF.  Using the content $(1, 2, n-3)$, the maximum ascent number is ${\frac{n-1}{2}}$; indeed, with the central vertex colored $1$, there is an ascent for each edge $e=(\frac{n+1}{2})i$ with vertex $i>\frac{n+1}{2}$. However, for the content $(2, 1, n-3)$, the maximum ascent number is ${\frac{n-1}{2}+2}$. This ascent number is achieved by coloring the central vertex with $2$, and two vertices with labels smaller than $\frac{n+1}{2}$ with color $1$, obtaining two additional ascents. Thus, the coefficient of $M_{(1, 2, n-3)}$ is a polynomial in $q$ of degree ${\frac{n-1}{2}}$, but the coefficient of $M_{(2, 1, n-3)}$ is a polynomial in $q$ of degree ${\frac{n-1}{2}+2}$, and so $X(K_{1,n-1};\mathbf{x}, q)$ is not symmetric.
\end{proof}

We also have the following general result on bipartite graphs.

\begin{proposition} Let $G$ be a connected labeled bipartite graph with an odd number of edges and unequal bipartition, that is, $G=A\cup B$ for two independent sets $A$ and $B$ with $|A|\neq |B|$. Then, the CQF of $G$ is not palindromic, and in particular, is not symmetric.
\end{proposition}

\begin{proof}
Let $a=|A|$ and $b=|B|$; it is given that $a\neq b$. The coefficient of $M_{(a, b)}$ in $X(G; \mathbf{x}, q)$ consists of a single term $q^{r}$, where $r$ equals the number of edges $ij$ such that $i\in A$ and $j\in B$ with $i<j$. Similarly, the coefficient of $M_{(b, a)}$ in $X(G; \mathbf{x}, q)$ consists of a single term $q^{s}$, where $s$ equals the number of edges $ij$ such that $i\in A$ and $j\in B$ with $i>j$. Since $r+s=|E(G)|$ is odd, it follows that $r\neq s$. As the coefficients of $M_{(a, b)}$ and $M_{(b, a)}$ differ, $X(G; \mathbf{x}, q)$ is not palindromic.
\end{proof}

\begin{corollary}\label{cor:nonsymmetric-tree} Let $T$ be a labeled tree with an even number of vertices with unequal bipartition. Then, the CQF of $T$ is not palindromic, and in particular, is not symmetric. 
\end{corollary}

With this in mind, we therefore conclude with the following question, which was recently answered in \cite{GPS24}. 
\begin{question}

For which labeled trees $T$ is $X(T; \mathbf{x},q)$ symmetric?
\end{question}

In \cite{GPS24}, the authors prove that the path graph $P_n$ with the natural labeling is the \emph{only} tree with symmetric CQF.

The reader may wonder how our paper and \cite{GPS24} differ and coincide. As noted earlier, both papers contain a version of Proposition~\ref{prop:LURL}. Our map in the proof of Proposition~\ref{prop:regular} is a specialization of \cite[Definition 4.7]{GPS24} with $k=1$. Beyond that, the proof methods are different, and all results were arrived at independently. It is interesting that in both papers,  enumerating the multiplicity of $q^{|E|}$ in a CQF is used to determine whether a CQF is symmetric.

\section{Statements and Declarations}

\textbf{Conflict of Interest Statement.} The authors collectively declare that there is no conflict of interest arising from this work.

\medskip 

\textbf{Data Availability Statement.} No data sets were used in the current manuscript.

\bibliographystyle{alpha}
\bibliography{biblio}

\end{document}